\documentclass[12pt,reqno]{amsart}
\usepackage{amsmath,amssymb,amsthm,pinlabel}
\usepackage[final,colorlinks=true]{hyperref}

\theoremstyle{plain}
\newtheorem{theorem}{Theorem}
\newtheorem{lemma}[theorem]{Lemma}
\newtheorem{proposition}[theorem]{Proposition}
\newtheorem{corollary}[theorem]{Corollary}

\newtheorem*{claim*}{Claim}

\theoremstyle{definition}

\newtheorem{remark}[theorem]{Remark}

\newcommand{\BZ}{\mathbb{Z}}

\newcommand{\BS}{\mathbb{S}}

\newcommand{\calA}{\mathcal{A}}

\newcommand{\calC}{\mathcal{C}}
\newcommand{\calE}{\mathcal{E}}

\newcommand{\calS}{\mathcal{S}}

\newcommand{\calV}{\mathcal{V}}
\newcommand{\calX}{\mathcal{X}}

\newcommand{\te}{{\tilde{e}}}
\newcommand{\tne}{{\tilde{\text{\sc{e}}}}}
\newcommand{\nate}{\text{\sc{e}}}

\newcommand{\from}{\colon\thinspace}

\newcommand{\la}{\langle}
\newcommand{\ra}{\rangle}
\newcommand{\bd}{\partial}

\newcommand{\cv}{cv}

\newcommand{\Cylo}{Cyl}

\DeclareMathOperator{\Out}{Out}
\DeclareMathOperator{\cplx}{\xi}

\DeclareMathOperator{\rk}{rk}
\DeclareMathOperator{\vol}{vol}

\begin{document}


\title[Volume of periodic free factors]{An algorithm to detect\\full irreducibility by bounding the volume of periodic free factors}

\author[M.~Clay]{Matt Clay}
\address{Dept.\ of Mathematics \\
  University of Arkansas\\
  Fayetteville, AR 72701}
\email{\href{mailto:mattclay@uark.edu}{mattclay@uark.edu}}

\author[J.~Mangahas]{Johanna Mangahas}
\address{Dept.\ of Mathematics \\
  Brown University\\
  Providence, RI 02912}
\email{\href{mailto:mangahas@math.brown.edu}{mangahas@math.brown.edu}}

\author[A.~Pettet]{Alexandra Pettet}
\address{Dept.\ of Mathematics \\
  University of British Columbia\\
  Vancouver, BC V6T 1Z2}
\email{\href{mailto:alexandra@math.ubc.ca}{alexandra@math.ubc.ca}}

\thanks{\tiny The first author is partially supported by NSF grant
  DMS-1006898. The second author is partially supported by NSF award 
  DMS-1204592. The third author is partially supported by an NSERC
  Discovery grant.}

\begin{abstract} 
  We provide an effective algorithm for determining whether an element
  $\phi$ of the outer automorphism group of a free group is fully
  irreducible. Our method produces a finite list which can be checked
  for periodic proper free factors. 
\end{abstract}

\maketitle


\section{Introduction}

Let $F$ be a finitely generated nonabelian free group of rank at least
2. An outer automorphism $\phi$ is \emph{reducible} if there exists a
free factorization $F = A_1 \ast \cdots A_k \ast B$ such that $\phi$
permutes the conjugacy classes of the $A_i$; else it is
\emph{irreducible}.  Although irreducible elements have nice
properties, e.g., they are known to possess irreducible train-track representatives, 
irreducibility is not preserved under iteration.  Thus one often considers
elements that are \emph{irreducible with irreducible powers
  \textup{(}iwip\textup{)}}, or \emph{fully
  irreducible}.  These are precisely the outer automorphisms $\phi$
for which there does not exist a proper free factor $A < F$ whose
conjugacy class $[A]$ satisfies $\phi^p([A]) = [A]$ for any $p > 0$.
If $\phi^p([A]) = [A]$ for some proper free factor $A < F$ and for
some $p>0$, we say $[A]$ is $\phi$--periodic, and, to avoid cumbersome language, also that the free factor $A$ is
$\phi$--periodic.  Fully irreducible elements are considered analogous
to pseudo-Anosov mapping classes of hyperbolic surfaces. As such, they
play an important role in the geometry and dynamics of the outer
automorphism group $\Out(F)$ of $F$.

Although considered in some sense a ``generic'' property in $\Out(F)$,
full irreducibility is not generally easy to detect. Kapovich
\cite{un:Kapovich} gave an algorithm for determining whether a given
$\phi \in \Out(F)$ is fully irreducible, inspired by Pfaff's criterion for
full irreducibility in \cite{un:Pfaff}. At points in his
algorithm, two processes run simultaneously, and although it is known
that one of these must terminate, it is not \emph{a priori} known
which will; it thus seems unclear that the complexity of Kapovich's
algorithm can be found without running the algorithm itself.

For mapping class groups and braid groups, there exist algorithms for
determining whether or not a given element is pseudo-Anosov
\cite{ar:Bestvina-Handel95,ar:Chen-Hamidi-Tehrani,un:Bell,ar:BernadeteGutierrezNitecki,ar:Los,un:Calvez}. Recently,
Koberda and the second author \cite{un:Koberda-Mangahas} provided an
elementary algorithm for determining whether or not a given mapping
class is pseudo-Anosov, using a method of ``list and check.'' They
show that if a mapping class $f$ is reducible, i.e., has an invariant
multicurve, then the curves in its reduction system have length
bounded by an exponential function in terms of the number of
generators needed to write $f$.  Therefore, given a mapping class $f$,
a \textbf{list} is produced of all multicurves whose curves are
sufficiently short. The action of $f$ is then \textbf{checked} on
these finitely many multicurves. If $f$ fixes a multicurve from the
list, it is reducible; otherwise, it is necessarily pseudo-Anosov.

In this article, we provide, in essence, a method of ``list and check'' for
elements of $\Out(F)$, akin to that of Koberda and the second author.
That is, we provide an algorithm which, given an element $\phi$
expressed as a product of generators from a finite generating set of
$\Out(F)$, produces a finite list of free factors and checks each for
$\phi$-periodicity. The algorithm effectively determines whether or
not the given element $\phi$ is fully irreducible. By
\emph{effective}, we understand that there is a computable function which bounds the number of steps in terms of the size of the input and that does not utilize the algorithm. In particular, we avoid the use of dual processes, one of which must terminate.

\subsection*{Acknowledgements} 

We would like to thank the Centre de Recerca Matem\`atica for its
hospitality during its research program, \emph{Automorphisms of Free
  Groups: Algorithms, Geometry, and Dynamics}, in Fall, 2012.  The
authors also thank Sam Taylor for comments on an earlier version of
this work.

\section{Statement of Results}\label{sec:results}

By $\rk(F)$ we denote the rank of the free group $F$.  Let $\cplx(F) =
3\rk(F) - 3$.  This is the maximum number of edges in a finite graph with fundamental group $F$ and without degree one or two vertices. This is also the 
maximum number of isotopy classes of disjoint, essential (not
bounding a ball) spheres in the double of the handlebody of genus
$\rk(F)$.  An element $\phi \in \Out(F)$ that is not fully irreducible
is \emph{cyclically reducible} if there exists a $\phi$--periodic rank
1 free factor; else it is \emph{noncyclically reducible}.

Our algorithm to determine full irreducibility of an element $\phi \in
\Out(F)$ consists of two effective processes. Process I determines (in
the absence of an obvious reduction) if $\phi$ is cyclically
reducible.  As we shall see in Section \ref{sec:listandcheck}, this
will exploit algorithms which are already well-known. Our main
contribution to the algorithm is in process II.  For this we construct
a finite list of conjugacy classes of proper free factors that
contains a $\phi$--periodic free factor if $\phi$ is noncyclically
reducible.  The length of this list is controlled by the word length
of $\phi$; this is the content of Theorem \ref{thm:main}. A systematic
check of the list then determines whether or not $\phi$ is fully
irreducible.

To state our main theorem, we start by fixing a basis $\calX$ for the
free group $F$, and let $T = T_\calX$ denote the Cayley graph for $F$
with respect to $\calX$. Given a subgroup $A \leq F$, the
\emph{volume} $\| A \|_\calX$ of $A$ is the number of edges in the
\emph{Stallings core} of the graph $T/F$. Recall that the Stallings
core is the graph $T_A/A$, where $T_A$ is the minimal subtree of $T$
with respect to the action of $A$; or, equivalently, the Stallings
core is the smallest subgraph of the cover of $T/F$ associated to $A$
that contains every embedded cycle (see~\cite{ar:Stallings} for
details). Note that the volume function $\| \cdot \|_\calX$ is
constant on conjugacy classes of subgroups. The quantity $\| A
\|_\calX$ gives some measure of the complexity of the subgroup $A$ in
terms of the basis $\calX$. For instance, if $A = \la a \ra$ is a
cyclic subgroup, the volume $\| \la a \ra \|_\calX$ is the cyclic
length of the element $a$ as a word in the basis $\calX$.

Now fix a finite generating set $\calS$ for $\Out(F)$. Denote by
$\vert \phi\vert_\calS$ the word length of $\phi \in \Out(F)$ with
respect to $\calS$. Our main theorem describes a relation between the
word length of a noncyclically reducible element of $\Out(F)$ and the
volume of one of its periodic free factors.

\begin{theorem}\label{thm:main}
There is a computable constant $C = C(\calX,\calS)$ such that, for any $\phi \in \Out(F)$, either
\begin{itemize}
\item[(i)] $\phi$ is fully irreducible, or
\item[(ii)] $\phi$ has a periodic rank--1 free factor, or
\item[(iii)] $\phi$ has a periodic proper free factor $A$ such that $\| A \|_\calX \leq C^{|\phi|_\calS}$.
\end{itemize}
\end{theorem}

In other words, if $\phi$ is noncyclically reducible, then $C^{|\phi|_\calS}$ bounds the volume of some proper $\phi$--periodic free factor.  An exact formula for $C$ is given at the end of
Section~\ref{sec:main}.

As there are a finite number of conjugacy classes of free
factors $A$ of $F$ for which $\| A \|_\calX$ is bounded, the theorem
provides a bound for the size of a list of conjugacy classes of free
factors that can be used to conclusively determine whether or not an
element $\phi \in \Out(F)$ of length $| \phi |_\calS$ is fully
irreducible, if $\phi$ is not cyclically reducible.

To prove Theorem~\ref{thm:main}, we utilize a notion of
\emph{intersection number} $i(S,T)$ defined between a pair of trees
$S$ and $T$ equipped with an isometric action by $F$, as defined by
Guirardel~\cite{ar:Guirardel}.  Horbez \cite{un:Horbez} related the
intersection number $i(T,T\phi)$ to the word length of $\phi \in
\Out(F)$ (Section~\ref{sec:intersection}, Theorem \ref{thm:Horbez}). We thus need only bound the
volume of a $\phi$--periodic proper free factor by $i(T,T\phi)$
(Section \ref{sec:main}, Proposition \ref{prop:max-fixed-length}).

Before embarking on the details of the proof of Theorem
\ref{thm:main}, we will first describe the procedure used in our
algorithm for detecting fully irreducible elements of $\Out(F)$. This
is contained in the next section, where we establish:
\begin{theorem}\label{thm:iwip-algorithm}
  There exists an effective algorithm for determining if an outer
  automorphism is fully irreducible.
\end{theorem}


\section{List and check algorithm}\label{sec:listandcheck}

The input of our algorithm is an element $\phi_0 \in \Out(F)$. Recall
that $\phi_0 \in \Out(F)$ is not fully irreducible if there exists a
periodic proper free factor, and note that the periodic free factors
of $\phi_0$ are exactly the periodic free factors of each of its
powers.  Feighn and Handel \cite{ar:Feighn-Handel} showed that there
is a power $Q$, depending on the rank of $F$ but not on the element
$\phi_0$, so that any periodic free factor of $\phi_0^Q$ is in fact
invariant.  An explicit function for $Q$ depending only on $\rk(F)$ can be found in \cite{un:HM-II} and \cite{un:FH14}.  For instance, Handel--Mosher show that this property is shared by all elements in $\mathrm{ker}(\Out(F) \to GL(\rk(F),\BZ_3))$ and hence $Q = \prod_{j=1}^{\rk(F)} (3^{\rk(F)} - 3^{j-1})$ suffices.  This is analogous to the fact
that the mapping class group has a finite index subgroup all of whose
elements are \emph{pure}; i.e., any invariant multicurve is curve-wise
fixed. As a preliminary step to our algorithm, we replace the element
$\phi_0$ by $\phi = \phi_0^Q$, so that henceforth we need only look
for invariant free factors.  Note that $\phi$ is irreducible if and
only if it is fully irreducible if and only if $\phi_0$ is fully
irreducible.

\medskip 

\noindent {\bf Process I.} To begin process I, we apply an effective algorithm due to Bestvina and Handel \cite{ar:Bestvina-Handel92} which finds a \emph{relative train track}\footnote{Loosely speaking, a relative train track representative is akin to a Jordan form for a linear transformation; we will not make use of any properties of relative train track representatives and refer the reader to the references for details.} representative $f\from \Gamma \to \Gamma$ of
$\phi$. At its conclusion, if $\Gamma$ has a nontrivial $f$--invariant
subgraph, then $\phi$ fixes a proper free factor and is therefore
reducible. Otherwise, the algorithm gives us an honest train track map
representing $\phi$. Recall that Bestvina and Handel
\cite{ar:Bestvina-Handel92} proved that the fixed subgroup of an
automorphism whose outer class is irreducible is at most rank 1. Thus we next want to check for loops homotopically fixed by $f$, which
correspond to a fixed conjugacy classes of $\phi$, and then see whether
their corresponding elements generate a higher rank subgroup of $F$.

For this, we make use of an algorithm of Turner in \cite{ar:Turner}. For an outer automorphism $\phi$ with train track map $f\from \Gamma \to \Gamma$, Turner begins 
by describing a graph $D_f$ equipped with a graph map $D_f \to \Gamma$. 
The components of $D_f$ are in one-to-one correspondence with the fixed 
subgroups of the automorphisms in the outer class of $\phi$, so that, restricted to a component 
of $D_f$, the map $D_f \to \Gamma$ is the covering map corresponding 
to the fixed subgroup of one of the elements of the outer class of $\phi$. The algorithm 
provides an effective procedure for obtaining a finite subgraph $C_f$ of $D_f$ 
that carries the fundamental group of $D_f$. If any component of $C_f$ has 
rank greater than 1, then $\phi$ is reducible. Otherwise, Whitehead's algorithm 
provides an effective method for determining whether any component of 
$C_f$ corresponds to a primitive element. If one does, then $\phi$ is cyclically reducible. This marks the end of process I. At this point, we stop if we have 
found that $\phi$ is reducible, and we continue to process II if we have only 
managed to determine that $\phi$ is noncyclically reducible or fully irreducible. 

\medskip

\noindent {\bf Process II.} Theorem~\ref{thm:main} gives
an upper bound $V = C^{|\phi|_\calS}$ on the volume of the smallest
$\phi$--invariant free factor, if $\phi$ is noncyclically
reducible. (Recall, we have replaced our original input $\phi_0$ by
$\phi = \phi_0^Q$ for which invariance and periodicity are the same.)
There are a finite number of conjugacy classes of subgroups $H$ with
volume less than this bound, and these can be systematically listed,
since they correspond to core graphs made from at most $V$ edges,
where each edge is oriented and labeled by an element of $\calX$. For
a gross overestimate of the number of these, one has $V \cdot
(2\rk(F))^V\cdot B_{2V}$, where $B_n$, known as the $n$th \emph{Bell
  number}, counts the number of partitions of $n$ objects. In our case
this is equivalent to the number of ways one can glue the $2V$
endpoints of $V$ edges to obtain a graph. In particular, the number of
conjugacy classes is less than $V(8V^2\rk(F))^{V}$ as $B_{2V} \leq
(2V)^{2V}$.  Whitehead's algorithm is then used to eliminate conjugacy
classes which are not free factors. We obtain a \textbf{list} of
conjugacy classes of free factors that are \textbf{checked} (using,
say, Stallings's graph pull backs \cite{ar:Stallings}) one-by-one for
$\phi$--invariance. This process, and hence the algorithm, stops once
either an invariant free factor is identified, concluding with $\phi$
reducible, or once every item on the list is checked and found not to
be invariant, determining that $\phi$ is fully irreducible.

\medskip

This completes the proof of Theorem \ref{thm:iwip-algorithm}, with the
assumption of Theorem \ref{thm:main}. Now we proceed with the proof of
Theorem \ref{thm:main}.


\section{Outer Space, Trees, and Morphisms}\label{sec:outerspace}

For mapping class groups, the intersection number between curves on
the surface is in various contexts useful in comparison to distances
in, for instance, Teichm\"uller space or the complex of curves.
Similar methods have been emerging for $\Out(F)$ and its associated
spaces. Culler and Vogtmann's \emph{outer space} is the space $cv$
consisting of metric simplicial trees $T$ equipped with simplicial, free $F$--actions that are \emph{minimal} (meaning they leave no proper subtree invariant), up to isometry which commutes with the action. The action of $\Out(F)$ on $cv$ is defined by pre-composing
the free group action with the outer automorphism; this action is
therefore on the right.  In some contexts, it is convenient to
consider the projectivized outer space $CV$ in which the sum of the
lengths of the edges of the quotient $T/F$ is 1.  Outer space is
treated as the analogue for $\Out(F)$ of Teichm\"uller space; we refer
the reader to Vogtmann's survey \cite{ar:Vogtmann} for a more detailed
description.

For a tree $T \in cv$, we use $d_T(\cdot,\cdot)$ to denote the
metric on $T$, $\ell_T(\cdot)$ to denote the length of edges or paths,
and $\calE(T)$ to denote the set of edges.  We may consider edges oriented, depending on the context.  In the special case that $T$ has a single vertex orbit and unit length on every edge, we call $T$ a \emph{unit rose}.

Any pair of unit roses $S, T$ are related by a \emph{morphism} $f \from S \to T$, by which we mean a cellular $F$--equivariant map that linearly expands every edge of $S$ over a non-backtracking edge path of $T$.  The \emph{length of a morphism} $f \from S\to T$ is
\[\ell(f) = \max \{ \ell_T(f(s)) \mid s \in \calE(S) \},\]
and the \emph{length of $S$ in $T$} is
\[\ell_T(S) = \min \{ \ell(f) \mid f\from S \to T \mbox{ is a morphism} \}.\]
We use $\ell_T(S)$ instead of Lipschitz distance to simplify computations in the next section, but we remark that the two values are easily related by \cite[Lemma~2.4]{un:Horbez}.  If $f \from S \to T$ satisfies $\ell(f) = \ell_T(S)$, we say $f$ is \emph{length minimizing}.  In general, $\ell_T(S)$ and $\ell_S(T)$ are not equal, but it is known that the ratio of their logarithms is bounded away from zero, independently of $S$ and $T$, when both trees are unit roses (or more generally, in the ``thick part'' of $cv$) \cite{ar:A-KB12,ar:HM07}.  We do not require this fact in what
follows.  Instead, it is convenient to define
\[\lambda(S,T) = \max\{ \ell_T(S),\ell_S(T) \}.\]


\section{Intersection and the Guirardel Core}\label{sec:intersection}

The utility of the intersection number between curves on a surface is carried over to free groups via the so-called \emph{Guirardel core} $\calC(S \times T)$: a certain closed, $F$--invariant (with the diagonal action) cellular subset of the product $S \times T$ of trees in $S, T \in cv$. The \emph{intersection number} $i(S,T)$ is the covolume of $\calC(S \times T)$, that is, the sum of the areas of the 2-cells in $\calC(S \times T)/F$.  Often we may assume $S$ and $T$ are unit roses, in which case $i(S,T)$ simply counts the squares in $\calC(S \times T)/F$.

For our purpose, we do not need the full definition of $\calC(S \times T)$, for which we refer the reader to~\cite{ar:Guirardel,un:Horbez}.  Rather, we make use of two approaches to computing $i(S,T)$.  In one of these, intersection numbers are interpreted as the geometric intersection between sphere systems in the doubled handlebody. This connection is recalled in the proof of Lemma~\ref{lem:vol-upperbound}, where it is used.

The other approach is a simple criterion, given by Behrstock, Bestvina and the first author~\cite{ar:Behrstock-Bestvina-Clay}, for when two edges $s \in \calE(S)$, $t \in \calE(T)$ determine a square $s \times t$ in the core $\calC(S \times T)$.  For a tree $T \in \cv$, we let $\bd T$ denote its boundary; that is, equivalence classes of geodesic rays where two rays are equivalent if their images lie in a bounded neighborhood of one another.  An oriented
edge $t \in \calE(T)$ determines a subset $\Cylo^+_T(t) \subset \bd T$, its \emph{(forward) one-sided cylinder}, which consists of equivalence classes of geodesics that contain a representative whose image contains $t$ with the correct orientation. The complement of $\Cylo^+_T(t)$ in $\bd T$ will be denoted by $\Cylo^-_T(t)$; clearly $\Cylo^-_T(t) = \Cylo^+_T(\bar{t})$, where $\bar{t}$ is $t$ with the reverse orientation. We will typically not bother with specifying an orientation as we will consider both one-sided cylinders simultaneously.  For $S,T \in cv,$ there exists a canonical $F$-equivariant homeomorphism $\bd \from \bd S \to \bd T$, which is induced by any morphism $f \from S \to T$. 

\begin{lemma}[{\cite[Lemma~2.3]{ar:Behrstock-Bestvina-Clay}}]\label{lem:core-criteria}
  Let $S,T \in \cv$ and let $\bd \from \bd S \to \bd T$ denote the
  canonical $F$--equivariant homeomorphism.  Given two edges $s \in
  \calE(S)$ and $t \in \calE(T)$, the rectangle $s \times t$ is in the
  core $\calC(S \times T)$ if and only if each of the four subsets
  $\bd(\Cylo^{(\pm)}_S(s)) \cap \Cylo^{(\pm)}_T(t)$ is nonempty.
\end{lemma}

Let $S,T \in \cv$ and $t \in \calE(T)$. The \emph{slice} of the core
$\calC(S \times T)$ above $t$ is the set: 
\begin{equation*}
  \label{eq:slice}
  \calC_t = \{ s \in \calE(S) \mid s \times t \subset \calC(S \times
  T) \}.
\end{equation*}
Similarly define the slice $\calC_s = \{ t \in \calE(T) \mid s \times t
\subset \calC(S \times T) \}$ for $s \in \calE(S)$. A simple
application of Lemma~\ref{lem:core-criteria} can be used to describe
the slice.

\begin{lemma}[{\cite[Lemma~3.7]{ar:Behrstock-Bestvina-Clay}}]
\label{lem:span-slice}
Let $S,T \in \cv$ and suppose $f\from S \to T$ is a morphism. Given an
edge $t \in \calE(T)$ and a point $y$ in the interior of $t$, the
slice $\calC_t \subset S$ of the core $\calC(S \times T)$ is contained
in the subtree spanned by $f^{-1}(y)$.
\end{lemma}

As $F$ acts freely on the edges of $T$, for any point $y$ that is in
the interior of $t$, the subtree $\calC_t \times \{y\}$ embeds in the
quotient $\calC(S \times T)/F$. Similarly, for a point $x$ in the
interior of $s$, the subtree $\{x\} \times \calC_s$ embeds in the
quotient. Therefore, the intersection number $i(S,T)$ can be expressed
as: \begin{equation}
  \label{eq:intersection}
  i(S,T) = \sum_{e \in \calE(T/F)} \ell_T(\te)\vol(\calC_\te) =
  \sum_{e \in \calE(S/F)} \ell_S(\te)\vol(\calC_\te).
\end{equation}
where by $\te$ we denote any lift of the edge $e$ to $T$ or $S$
respectively and by $\vol(\cdot)$ we denote the sum of the lengths of
the edges in the respective slice.

As mentioned in Section~\ref{sec:results}, Horbez \cite{un:Horbez} has
recently given, for two trees in $cv$, a bound on their Guirardel
intersection number based on $F$--equivariant maps between them, which
in turn we can relate to the geometry of $\Out(F)$.  We require a more
precise formulation of his result than what is stated in \cite{un:Horbez}, and we need only consider intersection between unit roses:

\begin{theorem}[{Horbez~\cite{un:Horbez}}]\label{thm:Horbez}
  Let $S,T \in cv$ be unit roses.  Then
  \begin{equation*}
    i(S,T) \leq 2\rk(F)^3\lambda(S,T)^4.
  \end{equation*}
\end{theorem}

For the remainder of this section, we derive the statement above by a variation on the arguments in \cite[Section~2.2]{un:Horbez}.  Using trees rather than marked graphs, and keeping track of the precise dependence on the $F$--equivariant maps, we obtain inequalities not stated directly in \cite{un:Horbez} but needed for our applications.

\begin{remark}\label{rem:unit-roses}
  If $f \from S \to T$ is a morphism between unit roses, then $f$ is a
  bijection between the vertices of $S$ and the vertices of $T$.
  Indeed, this follows as $F$ acts freely and transitively on the
  vertex sets.
\end{remark}

\begin{lemma}[cf.~Lemmas 2.3 and 2.5 in \cite{un:Horbez}]\label{lem:inverse-basis}
  Suppose that $S,T \in cv$ are unit roses and $f \from S \to T$ is a length minimizing morphism.  If $v_0, v_1$ are vertices of $T$, then there exist vertices $u_0, u_1$ of $S$ such that $f(u_0)=v_0, f(u_1)=v_1,$ and
\[d_S(u_0,u_1) \leq \lambda(S,T)d_T(v_0,v_1) + \lambda(S,T)^2.\] 
 \end{lemma}

 \begin{proof}
Since the existence of $u_0$ and $u_1$ is clear by Remark~\ref{rem:unit-roses}, we need only prove the inequality.  Let $g \from T \to S$ be a length minimizing morphism.  For any vertices $u, u'$ in $S$, it is clear by concatenating the images of edges that
   \[d_S(gf(u),gf(u'))\leq\ell(g)d_T(f(u),f(u'))\leq\ell(g)\ell(f)d_S(u,u').\]
Now consider the geodesic edge path $q$ from $u$ to $u'' = gf(u)$.  There is a $w \in F$ for which $d_S(u,wu) = 1$ and whose axis intersects $q$ only at $u$.  Thus $d_S(u'',wu'') = 2d_S(u,u'') + 1$.  On the other hand,
\[d_S(u'',wu'')=d_S(gf(u),wgf(u)) = d_S(gf(u),gf(wu)) \leq \ell(g)\ell(f),\]
by our first observation and our choice of $w$.  These last two statements together imply $2d_S(u,gf(u)) \leq \ell(g)\ell(f)$.  Since $u$ was arbitrary,
\begin{align*}
d_S(u_0,u_1) &\leq d_S(gf(u_0),gf(u_1)) + d_S(u_0,gf(u_0)) + d_S(u_1,gf(u_1)) \\
&\leq \ell(g)d_T(v_0,v_1) + \ell(g)\ell(f),
\end{align*}
from which the lemma follows.
 \end{proof}

\begin{lemma}[cf.~Proposition 2.8 in \cite{un:Horbez}]\label{lem:vanishing-path}
  Suppose $S,T \in cv$ are unit roses and that $f \from S \to T$ is a
  length minimizing morphism.  Given any edge $t \in \calE(T)$ and a
  point $y$ in the interior of $t$, for any $x,x' \in f^{-1}(y)$ we
  have
  \begin{equation*}
    d_S(x,x') \leq 4\lambda(S,T)^2 + 2.
  \end{equation*}
\end{lemma}

\begin{proof}
Fix $v_0$ a vertex on one end of $t$, and let $u_0$ be the vertex of $S$ such that $f(u_0)=v_0$.  Consider $x \in f^{-1}(y)$.  Let $s \in \calE(S)$ be the edge that contains $x$.  Let $u_1$ be a vertex on one end of $s$, and set $v_1 = f(u_1)$.  Observe that $f(s)$ is a geodesic of length at most $\ell(f)$.  Because $f(s)$ contains both $v_0$ and $v_1$, $d_T(v_0,v_1)\leq \ell(f)$.  By Lemma \ref{lem:inverse-basis},
\[d_S(u_0,u_1) \leq \lambda(S,T)\ell(f) + \lambda(S,T)^2 \leq 2\lambda(S,T)^2.\]
Because $d_S(u_1,x)\leq 1$, $d_S(u_0,x)\leq 2\lambda(S,T)^2 + 1$.  The same is true replacing $x$ with $x'$, hence the conclusion.
\end{proof}

\begin{corollary}\label{co:slice-volume}
  Suppose $S,T \in cv$ are unit roses. Given an edge $t \in \calE(T)$,
  the diameter of the slice $\calC_t \subset S$ of the core $\calC(S
  \times T)$ is at most $4\lambda(S,T)^2$.
\end{corollary}

\begin{proof}
  Let $f \from S \to T$ be a length minimizing morphism.  Suppose $y$
  is a point in the interior of the edge $t$.  By
  Lemma~\ref{lem:span-slice} any two points in the slice $\calC_t$ are
  contained in a geodesic between points in $f^{-1}(y)$, which by Lemma~\ref{lem:vanishing-path} has length at most $4\lambda(S,T)^2 + 2$.  The endpoints of this geodesic are interior to edges, while the slice is a union of closed edges; in particular the slice must exclude the partial edges at each end of the geodesic.  Thus the two points in the slice have distance at most $4\lambda(S,T)^2$.
\end{proof}

\noindent \emph{Proof of Theorem \ref{thm:Horbez}} (cf.~Proposition 2.8 in \cite{un:Horbez}).  
  Fix an edge $t \in \calE(T)$ and a point $y$ in the interior of $t$.
  Let $f \from S \to T$ be a length minimizing morphism.  The
  cardinality of $f^{-1}(y)$ is at most $\rk(F)\ell(f)$.  By Lemma~\ref{lem:span-slice} and Corollary~\ref{co:slice-volume}, the slice $\calC_t$ is covered by
  the union of $\frac{1}{2}\bigl(\rk(F)\ell(f)\bigr)^2$ edge paths of
  length at most $4\lambda(S,T)^2$.  Thus $\vol(\calC_t) \leq
  2\rk(F)^2\lambda(S,T)^4$.  Hence by \eqref{eq:intersection}, we have
  $i(S,T) \leq 2\rk(F)^3\lambda(S,T)^4$ as claimed.
\qed


\section{Subgroups and volume bounds}\label{sec:intersection-subgroups}

Given trees $S,T \in cv$ and a nontrivial finitely generated subgroup
$A \leq F$, there exist nonempty subtrees $S_A \subset S$, $T_A
\subset T$, on each of which $A$ acts minimally. We can thus consider the
Guirardel core $\calC(S_A \times T_A)$ for these minimal subtrees with
respect to the action of $A$. We might hope that, if $A$ is a free
factor of $F$, then $\calC(S_A \times T_A)$ embeds into $\calC(S
\times T)$, so that $i(S_A,T_A)$ is always dominated by $i(S,T)$.
Unfortunately this appears to be too much to expect, but we do
achieve:

\begin{proposition}\label{prop:minimal-bound}
  Let $A$ be a noncyclic finitely generated subgroup of $F$. Suppose
  that $S,T \in \cv$ are unit roses and let $S_A \subset S$, $T_A
  \subset T$ be the minimal subtrees with respect to $A$.  Then:
  \begin{equation*}
    i(S_A,T_A) \leq 6\cplx(A) \cdot \lambda(S,T)^3 \cdot i(S,T).
  \end{equation*}
\end{proposition}

\begin{proof}
  By Equation (\ref{eq:intersection}), we have:
  \begin{align*}
    i(S,T) &= \sum_{e \in \calE(T/F)} \vol(\calC_\te)
    \end{align*}
  \begin{align}
    i(S_A,T_A) &= \sum_{e \in \calE(T_A/A)}
    \ell_{T_A}(\te)\vol(\calA_\te)\label{eq:i(SA,TA)}
  \end{align}
  where $\calC_\te \subset S$ and $\calA_\te \subset S_A \subset S$
  are the slices in the respective cores. We denote by 
  $\bd \from \bd S \to \bd T$, the canonical
  $F$--equivariant homeomorphism, and by $\bd_A \from \bd S_A
  \to \bd T_A$, the canonical $A$--equivariant homeomorphism. Observe 
  that $\bd \big|_{\bd S_A} = \bd_A$.

  First, we claim that for each edge $\te \subset T_A \subset T$ we
  have that $\calA_\te \subseteq \calC_\te$. Indeed, let $s$ be an
  edge in $\calA_\te$. By Lemma~\ref{lem:core-criteria}, each of the
  four sets $\bd_A(\Cylo^{(\pm)}_{S_A}(s)) \cap
  \Cylo^{(\pm)}_{T_A}(\te)$ is non-empty. As
  $\bd_A(\Cylo^{(\pm)}_{S_A}(s)) = \bd(\Cylo^{(\pm)}_{S_A}(s)) \subset
  \bd(\Cylo^{(\pm)}_S(s))$ and $\Cylo^{(\pm)}_{T_A}(\te) \subset
  \Cylo^{(\pm)}_T(\te)$, each of the four sets
  $\bd(\Cylo^{(\pm)}_S(s)) \cap \Cylo^{(\pm)}_T(\te)$ is nonempty.
  Hence $s$ is an edge in $\calC_\te$.

  By a \emph{natural edge} of $T_A$ we mean an edge path
  $\tne = \te_1,\te_2,\ldots,\te_n$ that is a connected component of $T_A -
  \calV_{\geq 3}(T_A)$, where $\calV_{\geq 3}(T_A)$ is the collection
  of vertices of degree at least three. A natural edge in $T_A/A$ is
  the image of a natural edge in $T_A$; the set of all natural edges
  is denoted $\calE_N(T_A/A)$.

  Suppose that $\tne$ is a natural edge of $T_A$ consisting of the edge
  path $\te_1,\te_2,\ldots,\te_n$. Then since
  $\Cylo^{(\pm)}_{T_A}(\te_i) = \Cylo^{(\pm)}_{T_A}(\te_j)$ for all
  $i,j$, from Lemma~\ref{lem:core-criteria} we see that $\calA_{\te_i}
  = \calA_{\te_j}$. Therefore, we are justified in writing
  $\calA_{\tne}$ to denote any of the slices $\calA_{\te_i}$. Applying
  the observation that $\calA_{\te_i} \subseteq \calC_{\te_i}$ we have
  that: 
  \begin{equation}\label{eq:natural-slice}
    \calA_\tne \subseteq \bigcap_{i=1}^n \calC_{\te_i}
  \end{equation}

Next let us bound $\ell_{T_A}(\tne)$ whenever $\calA_{\tne}$ is not empty.  Let $f \from S \to T$ be a length minimizing morphism, and consider interior points $x_1 \in \te_1$ and $x_n \in \te_n$.  If there exists a point $p \in \calA_{\tne} \subseteq \calC_{\te_1} \cap \calC_{\te_n}$, it lives in a geodesic between a pair of points in $f^{-1}(x_1)$, by Lemma~\ref{lem:span-slice}.  By Lemma~\ref{lem:vanishing-path}, this path has length at most $4\lambda(S,T)^2 + 2$.  The same is true replacing $x_1$ with $x_n$.  Because $p$ is in the intersection of these two bounded geodesics, we may choose $y_1 \in f^{-1}(x_1)$ and $y_n \in f^{-1}(x_n)$ so that $d(y_1,y_n)\leq 4\lambda(S,T)^2 + 2$.  Furthermore, $d(x_1,x_n) \leq \ell(f)d(y_1,y_n)$, and we may choose $x_1$ and $x_n$ so that $d(x_1,x_n)$ is arbitrarily close to $\ell_{T_A}(\te)$.  Thus we conclude that,  when $\vol(\calA_{\tne}) > 0$,  
\[\ell_{T_A}(\tne) \leq \ell_T(S)\bigl(4\lambda(S,T)^2 + 2\bigr)< 6\lambda(S,T)^3.\]

Rewriting \eqref{eq:i(SA,TA)}, we
  get:
  \begin{align*}\label{eq:i(SA,TA)2}
    i(S_A,T_A) &= \sum_{e \in \calE(T_A/A)}
    \ell_{T_A}(\te)\vol(\calA_\te) \\
    &= \sum_{\nate \in \calE_N(T_A/A)}
    \ell_{T_A}(\tne)\vol(\calA_\tne) \\
    & \leq \sum_{\nate \in \calE_N(T_A/A)} 6\lambda(S,T)^3\vol(\calA_\tne)
  \end{align*} 
  By \eqref{eq:natural-slice}, $\vol(\calA_\tne) \leq i(S,T)$
  for every natural edge $\nate \in \calE_N(T_A/A)$.  As $T_A/A$ has at most $\cplx(A)$ natural edges, the proof is completed.  
\end{proof}

In order to eventually relate intersection number to the volume of an
invariant free factor, we find an effective lower bound for
intersection under a bounded iterate of a fully irreducible
automorphism.

\begin{lemma}\label{lem:vol-upperbound} 
  Let $\phi$ be a fully irreducible element of $\Out(F)$ and consider
  a tree $T \in cv$ with edge lengths at least 1. Then for some $1
  \leq P \leq \cplx(F)$:
  \[ \vol(T/F) \leq \cplx(F) \cdot i(T,T\phi^P). \]
\end{lemma} 

\begin{proof} 
  The lemma will be proved once we establish that, for some $1 \leq P
  \leq \cplx(F)$, the slice of the core $\calC(T \times T\phi^P)$
  above the longest edge of $T$ contains at least one edge of
  $T\phi^P$ and hence has volume at least 1. This is because, if the
  longest edge is $\te$, by Equation~\eqref{eq:intersection} we would
  have:
  \[\vol(T/F) \leq
  \cplx(F) \cdot \ell_T(\te) \leq
  \cplx(F) \cdot \ell_T(\te)\cdot\vol(\calC_\te) \leq 
  \cplx(F) \cdot i(T,T\phi^P).\]

  To prove the claim above, we compute the intersection number using
  \emph{sphere systems} in the doubled handlebody with fundamental
  group $F$.  Briefly, let $M$ be the connect sum of as many copies of
  $S^1 \times S^2$ as the rank of $F$.  By $\BS$ we denote the
  simplicial complex whose $n$--simplicies correspond to $n+1$ isotopy
  classes of disjoint essential spheres in $M$, and by $\BS^\infty$ we
  denote the subcomplex of $\BS$ consisting of simplicies where the
  complement of the corresponding sphere system in $M$ has a
  non-simply-connected component.  By work of
  Laudenbach~\cite{ar:Laudenbach73,bk:Laudenbach74}, there is a
  well-defined simplicial action of $\Out(F)$ on $\BS$ that leaves
  $\BS^{\infty}$ invariant.  In this action, fully irreducible
  elements of $\Out(F)$ act on $\BS$ without periodic orbits.
  Hatcher~\cite{ar:Hatcher} established an $\Out(F)$--equivariant
  isomorphism between projectivized outer space $CV$ and $\BS~-~\BS^{\infty}$.  Under this isomorphism, edges of a marked graph
  $T/F$ correspond bijectively to spheres in some sphere system.

  Horbez details the correspondence between geometric intersection of
  the sphere systems and the volume of the Guirardel core
  \cite{un:Horbez}.  In particular he shows that, if $T_0, T_1$ are
  trees in $CV$, then for the corresponding sphere systems
  $\Sigma_0,\Sigma_1 \in \BS - \BS^{\infty}$, we have $i(T_0,T_1) =
  i(\Sigma_0,\Sigma_1)$, where the latter counts the minimal number of
  circles common to each sphere system, weighted appropriately.  This minimum is acheived by representative sphere systems in a notion of normal position first described by Hatcher in \cite{ar:Hatcher}. While
  not stated explicitly in \cite{un:Horbez}, it can be verified that
  each circle of intersection occurring on a given component $\sigma_0
  \in \Sigma_0$ corresponds to an edge in the slice of the core
  $\calC(T_0 \times T_1)$ above an edge in $T_0$ corresponding to the
  lift of the edge in $T_0/F$ dual to $\sigma_0$.  This is because $\calC(T_0 \times T_1)$ can be built as the 2-complex dual to preimages of $\Sigma_1$ and $\Sigma_2$ in the universal cover of $M$, where $\Sigma_1$ and $\Sigma_2$ are assumed to be in normal position.  For our
  considerations, the weights on the spheres do not matter as we are
  only concerned with showing that some slice is nonempty, i.e., that
  the corresponding sphere has nontrivial intersection with another
  sphere.

  Now, given $T \in cv$, we scale its edges equally by $1/\vol(T/F)$
  to get a point $\overline{T} \in CV$. The slice over the longest
  edge of $T$ in $\calC(T \times T\phi^P)$ is non-empty if and only if
  the slice over the longest edge of $\overline{T}$ in
  $\calC(\overline{T} \times \overline{T}\phi^P)$ is
  non-empty. Suppose $\Sigma$ is the sphere system dual to
  $\overline{T}$ and that $\sigma \in \Sigma$ is dual to the longest
  edge of $\overline{T}$.  As the maximum number of isotopy classes
  of disjoint essential spheres in $M$ is $\cplx(F)$ and as $\phi$ is
  fully irreducible, at least two of the spheres $\sigma,
  \phi(\sigma),\dots,\phi^{\cplx(F)}(\sigma)$ have essential
  intersection, so in particular $\sigma$ essentially intersects
  $\phi^P(\sigma)$ for some $1 \leq P \leq \cplx(F)$. By the foregoing
  discussion, this means the slice above the longest edge in $T/F$
  contains at least one edge, proving the lemma.
\end{proof}

\begin{remark}
\label{rem:via trees}
We use sphere systems in the proof above to potentially give intuition on how the reasoning parallels that for its mapping class group analogue in \cite{un:Koberda-Mangahas}.  Alternately, Lemma~\ref{lem:vol-upperbound} can be proved using trees instead; we give a sketch of that argument here.  As in Lemma~\ref{lem:vol-upperbound}, we will show that the slice of the core above the longest edge in $T$ contains at least one edge.  

To this end, let $\te$ be the longest edge in $T$ and consider $T_{1} = X$, the tree obtained by collapsing every edge other than the ones in the orbit of $\te$.  If $i(T_{1},X\phi) \neq 0$, then the slice $\calC_{\te} \subset X\phi$ is non-empty as $i(T_{1},X\phi) = \ell_{T_{1}}(\te)\vol(\calC_{\te})$.  As $T\phi$ collapses to $X\phi$, we see that the slice of the core above $\te$ in $\calC(T \times T\phi)$ is also non-empty.        

If $i(T_{1},X\phi) = 0$, then by~\cite[Theorem~6.1]{ar:Guirardel} the core is tree, denote it $T_{2}$.  Moreover, $T_{2}$ is a \emph{common refinement} for both $T_{1}$ and $X\phi$.  That is, there are edge collapse maps $T_{1} \leftarrow T_{2} \to X\phi$.  There is a unique edge in $T_{2}$ that is mapped homeomorphically to $\te \subset T_{1}$.  Abusing notation, we denote this edge by $\te$ as well.  

As before, if $i(T_{2},X\phi^{2}) \neq 0$, then we see that the slice $\calC_{\te} \subset X\phi^{2}$ is non-empty and therefore the slice above $\te$  in $\calC(T \times T\phi^{2})$ is also non-empty.  If $i(T_{2},X\phi^{2}) = 0$, then the core is a common refinement for the two trees $T_{2}$ and $X\phi^{2}$, denote it $T_{3}$.

Continue in this fashion, if $i(T_{k},X\phi^{k}) = 0$, denote the common refinement by $T_{k+1}$.  As $\phi$ is fully irreducible at each step $T_{k}/F$ has $k$ edges.  Thus for some $1 \leq P \leq \cplx(F)$, we must have that $i(T_{P},X\phi^{P}) \neq 0$.  Hence the slice of the core above $\te$ in $\calC(T \times T\phi^{P})$ contains at least one edge, proving the lemma.
\end{remark}

We apply the previous two results to prove: 

\begin{proposition}\label{prop:max-fixed-length} 
  Let $T=T_\calX$ be the Cayley graph with respect to the basis
  $\calX$, with all edges of unit length. If $\phi \in \Out(F)$ acts
  fully irreducibly on a proper free factor $A$ of rank at least 2,
  then for some $1 \leq P \leq \cplx(F)$\textup{:} \[\| A \|_\calX
  \leq 6\cplx(F)^2 \cdot \lambda(T,T\phi^P)^3 \cdot i(T,T\phi^P). \]
\end{proposition}

\begin{proof} The minimal tree $T_A \subset T$ of $A$ has natural
  edge-lengths at least $1$, and as such can be thought of as an
  element of $cv(A)$, the unprojectivized outer space for $A$. We can
  apply Lemma \ref{lem:vol-upperbound} to $T_A$ with its free
  $A$--action to obtain $P \leq \cplx(A)$ for which \[\| A \|_\calX =
  \vol(T_A/A) \leq \cplx(A) \cdot i(T_A,T_A\phi^P).\] The conclusion
  follows by applying Proposition \ref{prop:minimal-bound}, noting
  that $\cplx(A) \leq \cplx(F)$.
\end{proof}

\section{Proof of Theorem~\ref{thm:main}}\label{sec:main}

In this section we prove the key new result for our algorithm, which
is Theorem~\ref{thm:main}. We wish to show that if $\phi \in \Out(F)$
is noncyclically reducible there is a $\phi$--periodic free factor
whose volume is bounded above by an exponential function in terms of
the word length $| \phi |_\calS$. Since $\phi$ is noncyclically
reducible, there is a $\phi^Q$--invariant free factor $A$ for which $1
< \rk(A) < \rk(F)$, where $Q=Q(\rk(F))$ is the constant power mentioned
at the beginning of Section~\ref{sec:listandcheck}. We can assume that
$\phi^Q|_A$ is fully irreducible.

Now let us complete the proof of Theorem \ref{thm:main}.  Let $T =
T_\calX$ be the Cayley graph with respect to the basis $\calX$, with
all edges of unit length. Combining with Theorem~\ref{thm:Horbez} with
Proposition~\ref{prop:max-fixed-length}, we have some $1 \leq P \leq
\cplx(F)$ for which
\begin{align*}
  \| A \|_\calX &\leq 6\cplx(F)^2 \cdot \lambda(T,T\phi^{QP})^3 \cdot i(T,T\phi^{QP})\\
  &\leq 12 \cplx(F)^5 \cdot \lambda(T,T\phi^{QP})^7
\end{align*}

Let $\lambda_\calX(\calS) = \max\{\lambda(T,T\psi) \mid \psi \in
\calS\}$.  Since the length of a composition of morphisms is bounded
by the product of their lengths, we have $\lambda(T,T\phi^{QP}) \leq
\lambda_\calX(\calS)^{QP|\phi|_\calS}$.
 
The proof of Theorem \ref{thm:main} is complete with:
\begin{align*}
  \| A \|_\calX &\leq 12\cplx(F)^5 \cdot \lambda_\calX(\calS)^{7QP|\phi|_\calS} \\
  &\leq C^{|\phi|_\calS},
\end{align*}
where $C = 12\cplx(F)^5 \cdot \lambda_\calS(\calX)^{7Q\cplx(F)}$
depends only on $\calX$ and $\calS$.


\bibliographystyle{amsplain}
\bibliography{test}

\end{document}